\documentclass[12pt]{amsart}

\usepackage{amsmath, amscd, amssymb, latexsym}
\usepackage{graphicx}
\usepackage{hyperref}
\usepackage{cleveref}
\usepackage{mathptmx}
\usepackage{esint}

\theoremstyle{plain}

\newtheorem{theorem}{Theorem}[section]
\newtheorem{lemma}[theorem]{Lemma}
\newtheorem{proposition}[theorem]{Proposition}
\newtheorem{corollary}[theorem]{Corollary}

\newtheorem{remark}[theorem]{Remark}

\newtheorem{theoremx}{Theorem}

\numberwithin{equation}{section}
\allowdisplaybreaks[3]

\crefname{theorem}{Theorem}{Theorems}
\crefname{lemma}{Lemma}{Lemmas}
\crefname{proposition}{Proposition}{Propositions}
\crefname{corollary}{Corollary}{Corollaries}
\crefname{question}{Question}{Questions}
\crefname{conjecture}{Conjecture}{Conjectures}
\crefname{definition}{Definition}{Definitions}
\crefname{remark}{Remark}{Remarks}
\crefname{section}{Section}{Sections}
\crefname{figure}{Figure}{Figures}
\crefname{theoremx}{Theorem}{Theorems}

\newcommand{\Ric}{\operatorname{Ric}}
\newcommand{\rank}{\operatorname{rank}}
\newcommand{\Rm}{\operatorname{Rm}}

\newcommand{\tr}{\operatorname{trace}}

\begin{document}

\title{Constant Rank Theorems under Intermediate Ricci Curvature}

\author{Mijia Lai}
\address{School of Mathematical Sciences, Shanghai Jiao Tong University, Shanghai 200240, China}
\email{laimijia@sjtu.edu.cn}

\author{Guoqiang Wu}
\address{School of Science, Zhejiang Sci-Tech University, Hangzhou 310018, China}
\email{gqwu@zstu.edu.cn}

\begin{abstract}
We prove two constant rank theorems for the Ricci tensor in two geometric settings: manifolds with harmonic curvature and gradient Ricci solitons. In both cases, if the $k$-intermediate Ricci curvature satisfies $\Ric_k\geq0$ for some $1\leq k\leq n-2$, then the Ricci tensor has constant rank and its nullity distribution is parallel. In the complete case, the universal cover splits isometrically along the Ricci-null directions. We also show that the range $1\leq k\leq n-2$ is sharp: in every dimension $n\geq4$, nonnegative Ricci curvature alone does not imply the constant rank property in either setting.
\end{abstract}

\maketitle

\section{Introduction}

Constant rank theorems, also known as microscopic convexity principles, play an important role in the study of convexity properties of solutions to elliptic partial differential equations. The earliest constant rank results of this type go back to Caffarelli--Friedman \cite{CaffarelliFriedman} and, independently, Singer--Wong--Yau--Yau \cite{SingerWongYauYau}. They were subsequently extended to higher dimensions by Korevaar--Lewis \cite{KorevaarLewis}. Roughly speaking, such results assert that, under suitable structural conditions on the equation, a nonnegative Hessian cannot change rank: once a zero eigenvalue occurs, the degeneracy propagates throughout the connected domain.
A landmark application of this principle is due to Guan--Ma \cite{Guan-Ma}, who established sufficient conditions on the prescribed data for the existence and uniqueness of smooth strictly convex solutions to the intermediate Christoffel--Minkowski problem, equivalently proving the existence of a unique convex hypersurface with prescribed
$k$-th symmetric function of principal radii. Their proof relied on a delicate constant rank theorem for a fully nonlinear Hessian equation, and this work has since become a model for the use of such arguments in geometric analysis.
This principle has since been developed systematically for fully nonlinear elliptic and parabolic equations; see, e.g., \cite{BianGuan,CaffarelliGuanMa, ChenHu2013}.

Besides its applications to the strict convexity of solutions, the constant rank method has important geometric consequences. A natural geometric setting in which the constant rank principle arises is that of Riemannian manifolds with harmonic curvature. A Riemannian manifold $(M^n,g)$ is said to have harmonic curvature if the Riemann curvature tensor is divergence-free, i.e., $\delta\Rm=0$. This class contains, among others, Einstein manifolds, locally conformally flat manifolds with constant scalar curvature, and Riemannian products of Einstein manifolds. Harmonic-curvature metrics have been studied extensively from the viewpoints of rigidity and local and global classification; see, for instance, \cite{Derdzinski1980,Derdzinski1982,DerdzinskiShen}. From this point of view, constant rank phenomena for the Ricci tensor provide a natural source of additional rigidity. In this setting, Caffarelli, Guan and Ma \cite{CaffarelliGuanMa} applied their constant rank method to show that, under nonnegative sectional curvature, the Ricci tensor of a manifold with harmonic curvature has constant rank.

Motivated by the result of Caffarelli--Guan--Ma, we investigate whether their nonnegative sectional curvature assumption can be weakened. To formulate our result, we recall the notion of intermediate Ricci curvature. Given an integer $1\leq k\leq n-1$, one says that $(M^n,g)$ has nonnegative $k$-intermediate Ricci curvature, denoted by $\Ric_k\geq0$, if
\[
\sum_{\alpha=1}^{k}R(v,e_\alpha,v,e_\alpha)\geq0
\]
for every orthonormal $(k+1)$-frame $\{v,e_1,\ldots,e_k\}$. Thus $\Ric_1\geq0$ is precisely nonnegative sectional curvature, while $\Ric_{n-1}\geq0$ is the usual nonnegative Ricci curvature. The conditions $\Ric_k\geq0$ therefore provide a natural interpolation between sectional and Ricci curvature. Our first main result shows that the same constant rank phenomenon continues to hold under the weaker condition \(\Ric_k\geq0\) for any fixed \(1\leq k\leq n-2\), while the endpoint assumption \(\Ric\geq0\) is in general insufficient in dimensions \(n\geq4\). Moreover, the constant rank phenomenon is accompanied by a geometric splitting of the Ricci-null distribution.

\begin{theoremx}\label{thm:mainA}
Let $(M^n,g)$ be a connected Riemannian manifold with harmonic curvature, and suppose that $\Ric_k\geq0$ for some $1\leq k\leq n-2$. Then the Ricci tensor has constant rank on $M$. Moreover, its nullity distribution $\mathcal N=\ker\Ric$ is a parallel subbundle of $TM$.

If, in addition, $(M^n,g)$ is complete, then its universal cover splits isometrically as
\[
(\widetilde M,\widetilde g)
=
\mathbb R^m\times(N^{n-m},g_N),
\]
where $m=\dim\ker\Ric$ and the Ricci tensor of $N$ is positive definite.
\end{theoremx}

The range $1\leq k\leq n-2$ in Theorem~\ref{thm:mainA} is sharp within the hierarchy of intermediate Ricci curvatures. Indeed, the endpoint condition $\Ric_{n-1}=\Ric\geq0$ is no longer sufficient in dimensions $n\geq4$. This is in contrast with the three-dimensional situation, where Caffarelli--Guan--Ma showed that nonnegative Ricci curvature already implies the constant rank property for the Ricci tensor on a manifold with harmonic curvature. In higher dimensions, however, the additional sectional curvature information contained in $\Ric_k\geq0$, $k\leq n-2$, is essential. We exhibit a local family of four-dimensional metrics with harmonic curvature and nonnegative Ricci curvature for which the rank of the Ricci tensor drops along a hypersurface. Taking products with flat factors then yields examples in every dimension $n\geq4$.

\begin{theorem}\label{thm:counterexample}
For every $n\geq4$, there exists a connected Riemannian manifold $(M^n,g)$ with harmonic curvature and $\Ric\geq0$ such that the Ricci tensor does not have constant rank.
\end{theorem}

Theorem~\ref{thm:mainA} also has a direct consequence for the classification problem of four-dimensional manifolds with harmonic curvature. In a recent work, Derdzi\'nski \cite{Derdzinski2025} made a further step toward such a classification. He observed that all presently known compact four-manifolds with harmonic curvature belong locally to one of five classes: Einstein manifolds; conformally flat manifolds of constant scalar curvature; products of a one-dimensional manifold with a conformally flat three-manifold of constant scalar curvature; products of two surfaces of constant Gaussian curvature; and a family of metrics arising from a special construction on products of surfaces. Moreover, he proved that if a harmonic-curvature four-manifold lies outside these five classes, then on a dense open subset its Ricci tensor has four distinct eigenvalues. In suitable local coordinates, the metric and the Ricci tensor are simultaneously diagonal, and the remaining structure is constrained by a system of homogeneous polynomial equations. Thus the eigenvalue structure of the Ricci tensor plays a central role in the four-dimensional classification problem.

Under the additional assumption $\Ric_2\geq0$, Theorem~\ref{thm:mainA} gives a particularly simple description of all non-full-rank cases. Indeed, if $(M^4,g)$ is complete and has harmonic curvature with $\Ric_2\geq0$, then either $\Ric>0$ everywhere, or the universal cover splits off a nontrivial Euclidean factor. More precisely, we have the following consequence.

\begin{corollary}\label{cor:four-dimensional-classification}
Let $(M^4,g)$ be a complete Riemannian four-manifold with harmonic
curvature and $\Ric_2\geq0$. If the Ricci tensor is not positive
definite, then its universal cover is isometric to one of
\[
\mathbb R^4,\qquad
\mathbb R^2\times\mathbb S^2(c),\qquad
\mathbb R\times\mathbb S^3(c),
\]
where $c>0$.
\end{corollary}

\begin{proof}
By Theorem~\ref{thm:mainA}, the Ricci tensor has constant rank and
\[
\widetilde M=\mathbb R^{4-r}\times N^r,
\qquad
\Ric_N>0,
\]
where $r=\rank\Ric$.

If $r=0$, then $\Ric\equiv0$, and
Corollary~\ref{cor:null-jacobi} implies that the metric is flat.
Thus $\widetilde M\cong\mathbb R^4$.

The case $r=1$ is impossible, since every one-dimensional manifold has
identically vanishing Ricci curvature.

If $r=2$, then $N^2$ has positive constant Gaussian curvature. Since
$N^2$ is complete and simply connected, it is isometric to a round
sphere $\mathbb S^2(c)$ for some $c>0$.

Finally, suppose $r=3$. Then
\[
\widetilde M=\mathbb R\times N^3.
\]
The factor $N^3$ is complete and simply connected, has positive Ricci
curvature and constant scalar curvature, and its Ricci tensor is
Codazzi. In dimension three this implies that the Cotton tensor
vanishes, so $N^3$ is locally conformally flat. By the classification
theorem of Hu--Li--Simon \cite{HuLiSimon}, a complete locally
conformally flat three-manifold with constant scalar curvature and
nonnegative Ricci curvature has universal cover isometric to one of
\[
\mathbb S^3(c),\qquad
\mathbb R^3,\qquad
\mathbb R\times\mathbb S^2(c).
\]
Since $N^3$ is already simply connected and $\Ric_N>0$, the latter two
possibilities are excluded. Hence
\[
N^3\cong\mathbb S^3(c).
\]
\end{proof}

Consequently, within the class $\Ric_2\geq0$, all harmonic-curvature four-manifolds with nontrivial Ricci nullity already fall into the classical product types appearing in Derdzi\'nski's list. Thus, if a complete harmonic-curvature four-manifold with $\Ric_2\geq0$ lies outside the five known classes, then necessarily $\Ric>0$ everywhere. Combining this observation with Derdzi\'nski's theorem, one sees that on a dense open subset such a hypothetical new example must satisfy
\[
0<\lambda_1<\lambda_2<\lambda_3<\lambda_4,
\]
where $\lambda_1,\ldots,\lambda_4$ are the eigenvalues of the Ricci tensor. In this sense, Theorem~\ref{thm:mainA} reduces the remaining classification problem under $\Ric_2\geq0$ to the positive-Ricci, four-distinct-eigenvalue branch.

The second natural setting for the constant rank problem is provided by gradient Ricci solitons. Recall that a gradient Ricci soliton is a triple $(M^n,g,f)$ satisfying
\[
\Ric+\nabla^2f=\lambda g
\]
for some constant $\lambda\in\mathbb R$. The geometry and rigidity of gradient Ricci solitons under curvature assumptions have been extensively studied. Petersen and Wylie \cite{PetersenWylieClassification,PetersenWylieRigidity} developed a general rigidity theory for gradient Ricci solitons and obtained classification and splitting results under various curvature conditions. In a direction particularly relevant to the present paper, Petersen--Wylie \cite{PetersenWylieClassification} and Guan--Lu--Xu \cite{GuanLuXu} proved that if a gradient Ricci soliton has nonnegative sectional curvature, then its Ricci tensor has constant rank; in the complete case, the universal cover splits isometrically along the nullity of the Ricci tensor. In the broader context of Bochner-type rigidity, Petersen and Wink \cite{PetersenWink2021} extended the Bochner technique to curvature operators, showing that suitable lower bounds on the averaged lowest eigenvalues of the curvature operator yield vanishing theorems for Betti numbers; in a sequel \cite{PetersenWink2022} they established Tachibana-type rigidity theorems in the K\"ahler and quaternionic K\"ahler settings. More recently, Li and Petersen \cite{LiPetersen2026} lowered the required curvature-operator threshold to $\frac{2(n-1)}{3}$
  for Einstein manifolds. While these results concern the curvature operator rather than the Ricci tensor, the underlying philosophy that a suitable positivity condition forces a degenerate algebraic structure to propagate is the same as the one pursued here, and the refined Ricci-algebra estimates of \cite{LiPetersen2026} are used in our companion discussion to relax the corresponding dimension thresholds.
 More recently, Wu--Wu \cite{Wu-Wu} studied rigidity and compactness of shrinking gradient Ricci solitons under curvature-pinching conditions weaker than nonnegative sectional curvature. Constant rank is also closely related to rigidity in the presence of constant scalar curvature: Fern\'andez-L\'opez and Garc\'ia-R\'io \cite{FernandezLopezGarciaRio} showed that, for complete gradient Ricci solitons with constant scalar curvature, the constant-rank property of the Ricci operator characterizes rigidity. For further work on constant-rank methods and the classification of Ricci solitons, see \cite{OuQuWu2025,Wu-Wu2}.

Our second main result extends the constant rank and splitting results of Petersen--Wylie and Guan--Lu--Xu. We show that their conclusion continues to hold under the weaker assumption $\Ric_k\geq0$ for any fixed $1\leq k\leq n-2$.

\begin{theoremx}\label{thm:mainB}
Let $(M^n,g,f)$ be a connected gradient Ricci soliton satisfying
\[
\Ric+\nabla^2 f=\lambda g
\]
for some constant $\lambda\in\mathbb R$. Suppose that
$
\Ric_k\geq0
$
for some $1\leq k\leq n-2$. Then the Ricci tensor has constant rank on $M$. Moreover, its nullity distribution
$
\mathcal N=\ker\Ric
$
is a parallel subbundle of $TM$.

If, in addition, $(M^n,g)$ is complete and
$
m=\dim\ker\Ric,
$
then its universal cover splits isometrically as
\[
(\widetilde M,\widetilde g)
=
\mathbb R^m\times(N^{n-m},g_N),
\]
where $\Ric_{g_N}>0$. Furthermore, the lift of the soliton potential splits as
\begin{equation}\label{eq:potential-splitting}
\widetilde f(x,y)
=
\frac{\lambda}{2}|x|^2+\langle a,x\rangle+f_N(y)+C.
\end{equation}
for some $a\in\mathbb R^m$ and $C\in\mathbb R$, where
\[
\Ric_{g_N}+\nabla_N^2f_N=\lambda g_N.
\]
\end{theoremx}

\begin{remark}
Related rigidity and compactness results for shrinking gradient Ricci
solitons under curvature conditions weaker than nonnegative sectional
curvature were obtained by Wu--Wu \cite{Wu-Wu}. In particular,
nonnegative second-intermediate Ricci curvature implies the curvature
pinching condition considered there. We emphasize that
Theorem~\ref{thm:mainB} is a local constant-rank statement and applies to shrinking,
steady, and expanding gradient Ricci solitons.
\end{remark}

We also construct local examples in every dimension $n\geq4$ showing that $\Ric\geq0$ is no longer sufficient to imply the constant rank property.

\begin{theorem}\label{thm:soliton-sharpness}
For every $n\geq4$, there exists a connected gradient shrinking Ricci
soliton $(M^n,g,f)$ with nonnegative Ricci curvature,
such that the Ricci tensor does not have constant rank.
\end{theorem}

We briefly explain the main ideas behind Theorem~\ref{thm:mainA} and Theorem~\ref{thm:mainB}. The first point is an elementary algebraic lemma
showing that nonnegative intermediate Ricci curvature controls the negative part of the sectional curvature in terms of the Ricci curvature itself. More precisely, if $2\leq k\leq n-2$ and $\Ric_k\geq0$, then for every orthonormal pair $X,Y$ one has
\[
K(X,Y)\geq-\frac{k-1}{n-1-k}\Ric(X,X).
\]
Let $0\leq\lambda_1\leq\cdots\leq\lambda_n$ be the eigenvalues of the Ricci tensor and set $\sigma_\ell=\lambda_1+\cdots+\lambda_\ell$. Since harmonic curvature is equivalent to the Codazzi condition for the Ricci tensor, one has the elliptic identity
\[
\Delta R_{ij}=R_{il}R_{jl}-R_{ikjl}R_{kl}.
\]
Combining this identity with the preceding curvature estimate and a viscosity subtrace argument by Bryan-Ivaki-Scheuer~\cite{BryanIvakiScheuer}, we obtain a differential inequality of the form
\[
\Delta\sigma_\ell\leq C(n,k)R\,\sigma_\ell
\]
in the viscosity sense. The strong maximum principle then implies that the nullity of the Ricci tensor is constant. At a null direction $V$, the same intermediate Ricci curvature estimate implies $R(V,\cdot,V,\cdot)=0$, and a Bochner-type argument shows that $\ker\Ric$ is parallel. The global splitting in the complete case then follows from the de Rham decomposition theorem.

The proof of Theorem~\ref{thm:mainB} is parallel in spirit to that of the harmonic-curvature case, although the underlying elliptic equation is different. The Ricci tensor of a gradient Ricci soliton is generally not Codazzi, but it satisfies the weighted equation
\[
\Delta_f\Ric
=
2\lambda\Ric-2\,\Rm*\Ric.
\]
The same algebraic estimate for the negative part of the sectional curvature under $\Ric_k\geq0$ therefore yields a viscosity differential inequality for the lower Ricci subtraces. The strong maximum principle for the weighted Laplacian then gives the constant rank property, while the vanishing of the curvature in Ricci-null directions implies that the nullity distribution is parallel. In the complete case, the de Rham decomposition theorem gives the metric splitting, and the soliton equation determines the Gaussian form of the potential on the Euclidean factor.

The rest of the paper is organized as follows. In Section~2 we prove Theorem~\ref{thm:mainA}. We first establish the basic algebraic estimates associated with nonnegative intermediate Ricci curvature, and then combine them with the Codazzi equation and a viscosity constant rank argument. We also prove that the Ricci-null distribution is parallel and derive the splitting of the universal cover in the complete case. At the end of Section~2 we construct the doubly warped four-dimensional example described above and prove Theorem~\ref{thm:counterexample}. In Section~3 we prove Theorem~\ref{thm:mainB} and construct a local counterexample showing the sharpness of its curvature assumption.

\section{Constant rank theorem for manifolds with harmonic curvature}
\label{sec:constant-rank}

In this section we prove Theorem~\ref{thm:mainA}. Throughout the paper, we use the convention that the unit sphere has positive sectional curvature and write $K(X,Y)=R(X,Y,X,Y)$ for orthonormal vectors $X,Y$. We denote by $R=\tr_g\Ric$ the scalar curvature and use the rough Laplacian $\Delta=\tr_g\nabla^2$.

The proof consists of three steps. We first derive an elementary algebraic consequence of the condition $\Ric_k\geq0$, which controls the negative part of the sectional curvature by the Ricci curvature. We then combine this estimate with the Codazzi equation and a viscosity inequality for the lower subtraces of the Ricci tensor to obtain the constant rank property. Finally, we show that the Ricci-null distribution is parallel and apply the de Rham decomposition theorem.

We begin with the main algebraic observation.

\begin{lemma}\label{lem:intermediate-algebra}
Let $(M^n,g)$ satisfy $\Ric_k\geq0$ for some $1\leq k\leq n-2$. Then $\Ric\geq0$. Moreover, for every orthonormal pair $X,Y$,
\begin{equation}\label{eq:sectional-lower-bound}
K(X,Y)\geq-\frac{k-1}{n-1-k}\Ric(X,X),
\end{equation}
where the right-hand side is understood to be $0$ when $k=1$.
\end{lemma}

\begin{proof}
Fix a unit vector $X$, and let $e_1,\ldots,e_{n-1}$ be an orthonormal basis of $X^\perp$. Set $\kappa_\alpha=K(X,e_\alpha)$ and, after relabeling, assume
$
\kappa_1\leq\kappa_2\leq\cdots\leq\kappa_{n-1}.
$
The condition $\Ric_k\geq0$ implies
\begin{equation}\label{eq:k-smallest}
\kappa_1+\cdots+\kappa_k\geq0.
\end{equation}
In particular, $\kappa_k\geq0$, and hence $\kappa_\alpha\geq0$ for all $\alpha\geq k$. It follows immediately that
\[
\Ric(X,X)=\sum_{\alpha=1}^{n-1}\kappa_\alpha\geq0.
\]
Since $X$ was arbitrary, $\Ric\geq0$.

For $k=1$, \eqref{eq:sectional-lower-bound} is precisely the assumption of nonnegative sectional curvature. We therefore assume $k\geq2$. There is nothing to prove if $\kappa_1\geq0$. Suppose that $\kappa_1=-a<0$. By~\eqref{eq:k-smallest},
\[
\kappa_2+\cdots+\kappa_k\geq a.
\]
Since $\kappa_2\leq\cdots\leq\kappa_k$, we obtain
$
\kappa_k\geq\frac{a}{k-1}.
$
Consequently, $\kappa_\alpha\geq a/(k-1)$ for every $\alpha\geq k+1$. Using again~\eqref{eq:k-smallest}, we find
\[
\begin{aligned}
\Ric(X,X)
&=\sum_{\alpha=1}^{k}\kappa_\alpha
+\sum_{\alpha=k+1}^{n-1}\kappa_\alpha\\
&\geq\frac{n-1-k}{k-1}a.
\end{aligned}
\]
Thus
\[
a\leq\frac{k-1}{n-1-k}\Ric(X,X),
\]
which proves~\eqref{eq:sectional-lower-bound}.
\end{proof}

The following consequence will also be crucial in the proof of the splitting statement.

\begin{corollary}\label{cor:null-jacobi}
Under the assumptions of Lemma~\ref{lem:intermediate-algebra}, if $V\in\ker\Ric$, then
\begin{equation}\label{eq:null-sectional}
K(V,X)=0
\end{equation}
for every $X\perp V$. Equivalently,
\begin{equation}\label{eq:null-jacobi-tensor}
R(V,\cdot,V,\cdot)\equiv0.
\end{equation}
\end{corollary}

\begin{proof}
By~\eqref{eq:sectional-lower-bound}, $K(V,X)\geq0$ for every unit vector $X\perp V$, since $\Ric(V,V)=0$. On the other hand,
\[
0=\Ric(V,V)=\sum_{\alpha=1}^{n-1}K(V,e_\alpha)
\]
for any orthonormal basis of $V^\perp$. Hence every summand vanishes, proving~\eqref{eq:null-sectional}. Polarizing the quadratic form $X\mapsto R(V,X,V,X)$ gives~\eqref{eq:null-jacobi-tensor}.
\end{proof}

We recall next the elliptic equation satisfied by the Ricci tensor on a manifold with harmonic curvature. Since $\delta\Rm=0$, the second Bianchi identity implies that the Ricci tensor is Codazzi:
\begin{equation}\label{eq:ricci-codazzi}
\nabla_kR_{ij}=\nabla_iR_{kj}.
\end{equation}
Taking a trace and using the contracted Bianchi identity shows that the scalar curvature $R$ is constant. Commuting covariant derivatives in~\eqref{eq:ricci-codazzi} gives
\begin{equation}\label{eq:ricci-laplacian}
\Delta R_{ij}=R_{il}R_{jl}-R_{ikjl}R_{kl}.
\end{equation}

For a symmetric $2$-tensor $A$, let $\lambda_1(A)\leq\cdots\leq\lambda_n(A)$ denote its eigenvalues and define $\sigma_\ell(A)=\lambda_1(A)+\cdots+\lambda_\ell(A)$. We use the following consequence of the subtrace lemma in \cite{BryanIvakiScheuer}.

\begin{lemma}\label{lem:subtrace-viscosity}
Let $A$ be a smooth symmetric $2$-tensor on $M$, and fix $p\in M$.
Choose an orthonormal eigenbasis $e_1,\ldots,e_n$ of $A$ at $p$, with
eigenvalues $\lambda_1\leq\cdots\leq\lambda_n$. If $\xi$ is a $C^2$
lower support function for $\sigma_\ell(A)$ at $p$, then for every
$X\in T_pM$,
\begin{equation}\label{eq:subtrace-gradient}
\nabla_X\xi(p)
=
\sum_{i=1}^{\ell}(\nabla_XA)(e_i,e_i),
\end{equation}
and
\[
\Delta\xi(p)\leq\sum_{i=1}^{\ell}(\Delta A)(e_i,e_i).
\]
\end{lemma}

\begin{proof}
Extend $e_1,\ldots,e_n$ to local orthonormal vector fields
$E_1,\ldots,E_n$ by radial parallel transport from $p$. Then
\[
\xi(q)\leq \sigma_\ell(A)(q)
\leq \sum_{i=1}^{\ell}A(E_i,E_i)(q)
\]
near $p$, and equality holds at $p$. Hence the first derivatives agree
at $p$, which gives \eqref{eq:subtrace-gradient}. Taking second
derivatives along geodesics issuing from $p$ and summing over an
orthonormal basis yields
\[
\Delta\xi(p)\leq
\sum_{i=1}^{\ell}(\Delta A)(e_i,e_i).
\]
\end{proof}

We apply the lemma to $A=\Ric$. At a fixed point $p$, choose Ricci eigenvectors such that
\[
\Ric(e_i,e_j)=\lambda_i\delta_{ij},
\qquad
0\leq\lambda_1\leq\cdots\leq\lambda_n,
\]
where the nonnegativity follows from Lemma~\ref{lem:intermediate-algebra}. In this frame,~\eqref{eq:ricci-laplacian} gives
\begin{equation}\label{eq:ric-eigen-laplacian}
(\Delta\Ric)_{ii}
=
\lambda_i^2-\sum_{j=1}^nK_{ij}\lambda_j,
\end{equation}
where $K_{ij}=K(e_i,e_j)$.

Set $c_{n,k}:=(k-1)/(n-1-k)$. By Lemma~\ref{lem:intermediate-algebra}, $K_{ij}\geq-c_{n,k}\lambda_i$. Since every $\lambda_j\geq0$,
\begin{align}
-\sum_jK_{ij}\lambda_j
&\leq-\sum_{K_{ij}<0}K_{ij}\lambda_j\nonumber\\
&\leq c_{n,k}\lambda_i
\sum_{K_{ij}<0}\lambda_j\nonumber\\
&\leq c_{n,k}R\lambda_i.
\label{eq:weighted-curvature-estimate}
\end{align}
Moreover, $0\leq\lambda_i\leq R$, and hence $\lambda_i^2\leq R\lambda_i$. Combining this with~\eqref{eq:ric-eigen-laplacian} and~\eqref{eq:weighted-curvature-estimate}, we obtain
\begin{equation}\label{eq:ric-component-estimate}
(\Delta\Ric)_{ii}
\leq
\frac{n-2}{n-1-k}R\lambda_i.
\end{equation}

Consequently, Lemma~\ref{lem:subtrace-viscosity} yields the following.

\begin{proposition}\label{prop:subtrace-inequality}
For every $1\leq\ell\leq n$, the function $\sigma_\ell=\lambda_1+\cdots+\lambda_\ell$ satisfies
\begin{equation}\label{eq:main-subtrace}
\Delta\sigma_\ell
\leq
\frac{n-2}{n-1-k}R\,\sigma_\ell
\end{equation}
in the viscosity sense.
\end{proposition}

\begin{proof}
Let $\xi$ be any smooth lower support for $\sigma_\ell$ at $p$. By Lemma~\ref{lem:subtrace-viscosity} and~\eqref{eq:ric-component-estimate},
\[
\Delta\xi(p)
\leq
\sum_{i=1}^{\ell}(\Delta\Ric)_{ii}
\leq
\frac{n-2}{n-1-k}R
\sum_{i=1}^{\ell}\lambda_i
=
\frac{n-2}{n-1-k}R\,\sigma_\ell(p).
\]
This is precisely~\eqref{eq:main-subtrace} in the viscosity sense.
\end{proof}

We can now prove the constant rank assertion.

\begin{proof}[Proof of the constant rank part of Theorem~\ref{thm:mainA}]
If $R=0$, then $\Ric\geq0$ immediately implies $\Ric\equiv0$, and there is nothing to prove. We may therefore assume $R>0$.

Let
\[
m=\max_{x\in M}\dim\ker\Ric_x.
\]
If $m=0$, then $\Ric>0$ everywhere and the assertion follows. Suppose $m\geq1$, and choose $p\in M$ with $\dim\ker\Ric_p=m$. Then $\sigma_m(p)=0$, while $\sigma_m\geq0$ on $M$. By Proposition~\ref{prop:subtrace-inequality},
\[
\Delta\sigma_m
-
\frac{n-2}{n-1-k}R\,\sigma_m
\leq0
\]
in the viscosity sense. The strong minimum principle implies $\sigma_m\equiv0$ on the connected manifold $M$.

Thus $\dim\ker\Ric_x\geq m$ for every $x\in M$. By the definition of $m$, equality must hold everywhere. Hence $\Ric$ has constant rank.
\end{proof}

We next prove that the Ricci-null distribution is parallel. Since the Ricci tensor has constant rank, $\mathcal N:=\ker\Ric$ is a smooth vector subbundle of $TM$.

\begin{proposition}\label{prop:nullity-parallel}
Under the assumptions of Theorem~\ref{thm:mainA}, the distribution $\mathcal N=\ker\Ric$ is parallel.
\end{proposition}

\begin{proof}
Let $V$ be a local smooth section of $\mathcal N$. Thus $\Ric(V,\cdot)=0$, and in particular $\Ric(V,V)\equiv0$. Let $e_1,\ldots,e_n$ be a local orthonormal frame which is normal at a fixed point. Differentiating twice gives
\begin{align}
0
&=
\Delta\bigl(\Ric(V,V)\bigr)\nonumber\\
&=
(\Delta\Ric)(V,V)
+
4\sum_a(\nabla_{e_a}\Ric)(\nabla_{e_a}V,V)
+
2\Ric(\Delta V,V)\nonumber\\
&\qquad
+
2\sum_a\Ric(\nabla_{e_a}V,\nabla_{e_a}V).
\label{eq:nullity-bochner-1}
\end{align}
Since $V\in\ker\Ric$, we have $\Ric(\Delta V,V)=0$ and $\Ric(\nabla_{e_a}V,V)=0$. Differentiating the latter identity and summing over $a$ yields
\[
\sum_a(\nabla_{e_a}\Ric)(\nabla_{e_a}V,V)
=
-\sum_a\Ric(\nabla_{e_a}V,\nabla_{e_a}V).
\]
Substituting into~\eqref{eq:nullity-bochner-1}, we obtain
\begin{equation}\label{eq:nullity-bochner-2}
0
=
(\Delta\Ric)(V,V)
-
2\sum_a\Ric(\nabla_{e_a}V,\nabla_{e_a}V).
\end{equation}

By~\eqref{eq:ricci-laplacian},
\[
(\Delta\Ric)(V,V)
=
|\Ric(V,\cdot)|^2
-
R(V,e_i,V,e_j)R_{ij}.
\]
The first term vanishes because $V\in\ker\Ric$, while the second vanishes by Corollary~\ref{cor:null-jacobi}. Hence $(\Delta\Ric)(V,V)=0$. Equation~\eqref{eq:nullity-bochner-2} therefore gives
\[
\sum_a\Ric(\nabla_{e_a}V,\nabla_{e_a}V)=0.
\]
Since $\Ric\geq0$, every term on the left-hand side is nonnegative. Consequently, $\Ric(\nabla_XV,\nabla_XV)=0$ for every tangent vector $X$, and hence $\nabla_XV\in\ker\Ric$. Thus $\mathcal N$ is parallel.
\end{proof}

Now we can finish the proof of Theorem~\ref{thm:mainA}.

\begin{proof}[Proof of the splitting assertion in Theorem~\ref{thm:mainA}]
Assume now that $(M,g)$ is complete. The lift of $\mathcal N$ to the universal cover $(\widetilde M,\widetilde g)$ is a parallel distribution, and so is its orthogonal complement. By the de Rham decomposition theorem,
\[
(\widetilde M,\widetilde g)
=
(M_0^m,g_0)\times(N^{n-m},g_N),
\]
where $m=\dim\ker\Ric$.

For every $V\in TM_0$, Corollary~\ref{cor:null-jacobi} gives $R(V,\cdot,V,\cdot)=0$. In particular, all sectional curvatures of $M_0$ vanish. Since $M_0$ is complete and simply connected, $(M_0,g_0)\cong\mathbb R^m$. Finally, the kernel of the Ricci tensor is exactly the tangent bundle of the Euclidean factor, and hence $\Ric_{g_N}>0$. Therefore
\[
(\widetilde M,\widetilde g)
\cong
\mathbb R^m\times(N^{n-m},g_N),
\qquad
\Ric_{g_N}>0.
\]
\end{proof}

Lastly, we show that the range $1\leq k\leq n-2$ in Theorem~\ref{thm:mainA} cannot in general be extended to the endpoint $k=n-1$. Thus ordinary nonnegative Ricci curvature is not sufficient for the constant rank property in dimensions $n\geq4$.

We first construct a four-dimensional local example. Let
$
M_\varepsilon=(-\varepsilon,\varepsilon)\times\mathbb S^2\times\mathbb S^1
$
and consider a metric of the form
\begin{equation}\label{eq:double-warp}
g=dt^2+f(t)^2g_{\mathbb S^2}+h(t)^2d\theta^2,
\end{equation}
where $g_{\mathbb S^2}$ is the standard metric of sectional curvature $1$, and $f,h>0$. Set $a=f'/f$ and $b=h'/h$.

Let $e_0=\partial_t$, let $\widetilde e_1,\widetilde e_2$ be an
orthonormal frame for the standard sphere, and set
\[
e_A=\frac{1}{f}\widetilde e_A,\quad A=1,2,
\qquad
e_3=\frac{1}{h}\partial_\theta .
\]
By the standard curvature formulas for a multiply warped product, the
nontrivial sectional curvatures are

\begin{align}
K_{0A}&=-\frac{f''}{f}=-(a'+a^2),\qquad A=1,2,
\label{eq:K0A}\\
K_{12}&=\frac{1-(f')^2}{f^2}=f^{-2}-a^2,\notag\\
K_{03}&=-\frac{h''}{h}=-(b'+b^2),\notag\\
K_{A3}&=-\frac{f'h'}{fh}=-ab,\qquad A=1,2.
\label{eq:KA3}
\end{align}

It follows that the Ricci tensor is diagonal with respect to this splitting. Denote its three distinct eigenvalue functions by $\rho_0,\rho_1,\rho_3$, where $\rho_1$ has multiplicity two. From~\eqref{eq:K0A}--\eqref{eq:KA3}, we obtain
\begin{align}
\rho_0&=-2(a'+a^2)-(b'+b^2),
\label{eq:rho0}\\
\rho_1&=-a'-2a^2-ab+f^{-2},\notag\\
\rho_3&=-b'-b^2-2ab.
\label{eq:rho3}
\end{align}

We next identify the harmonic-curvature condition.

\begin{lemma}\label{lem:codazzi-warp}
For the metric~\eqref{eq:double-warp}, the Ricci tensor is Codazzi if and only if
\begin{align}
\rho_1'&=a(\rho_0-\rho_1),
\label{eq:codazzi-rho1}\\
\rho_3'&=b(\rho_0-\rho_3).
\label{eq:codazzi-rho3}
\end{align}
Consequently,~\eqref{eq:codazzi-rho1}--\eqref{eq:codazzi-rho3} are equivalent to $g$ having harmonic curvature.
\end{lemma}

\begin{proof}
The connection formulas for the multiply warped product give
\[
\nabla_{e_A}e_0=ae_A,
\qquad
\nabla_{e_3}e_0=be_3.
\]
Since the Ricci tensor is diagonal and its eigenvalues depend only on $t$,
\[
(\nabla_{e_0}\Ric)(e_A,e_A)=\rho_1',
\qquad
(\nabla_{e_A}\Ric)(e_0,e_A)=a(\rho_0-\rho_1).
\]
Thus the Codazzi equation in the $e_0,e_A$ directions is precisely~\eqref{eq:codazzi-rho1}. Similarly,
\[
(\nabla_{e_0}\Ric)(e_3,e_3)=\rho_3',
\qquad
(\nabla_{e_3}\Ric)(e_0,e_3)=b(\rho_0-\rho_3),
\]
which gives~\eqref{eq:codazzi-rho3}. All remaining Codazzi equations are automatic by the symmetry of the warped product. Finally, by the second Bianchi identity, the Ricci tensor is Codazzi if and only if the Riemann curvature tensor is divergence-free.
\end{proof}

Introduce $u=a'$ and $v=b'$. Substituting~\eqref{eq:rho0}--\eqref{eq:rho3} into~\eqref{eq:codazzi-rho1}--\eqref{eq:codazzi-rho3} and simplifying gives
\begin{equation}\label{eq:ODE-system}
\left\{
\begin{aligned}
f'&=af,\\
h'&=bh,\\
a'&=u,\\
b'&=v,\\
u'&=-3au-bu-af^{-2}-a^2b+ab^2,\\
v'&=2a^2b-2ab^2-2(a+b)v.
\end{aligned}
\right.
\end{equation}
The right-hand side is smooth on the open set $\{f>0,\ h>0\}$. Hence arbitrary initial data with $f,h>0$ determine a unique local smooth solution.

We choose
\begin{equation}\label{eq:initial-data}
f(0)=h(0)=1,
\qquad
a(0)=b(0)=0,
\qquad
u(0)=1,
\qquad
v(0)=-3.
\end{equation}
By the standard local existence and uniqueness theorem for ordinary differential equations, there exists $\varepsilon>0$ and a unique smooth solution of~\eqref{eq:ODE-system} on $(-\varepsilon,\varepsilon)$ satisfying~\eqref{eq:initial-data}. Shrinking $\varepsilon$ if necessary, we may assume $f(t)>0$ and $h(t)>0$.

By construction and Lemma~\ref{lem:codazzi-warp}, the resulting metric has harmonic curvature. At $t=0$, formulas~\eqref{eq:rho0}--\eqref{eq:rho3} give
\begin{equation}\label{eq:ricci-initial}
\rho_0(0)=1,
\qquad
\rho_1(0)=0,
\qquad
\rho_3(0)=3.
\end{equation}
Thus the Ricci eigenvalues at $t=0$ are $1,0,0,3$.

It remains to show that the double zero eigenvalue becomes positive away from $t=0$. From~\eqref{eq:codazzi-rho1},
$
\rho_1'=a(\rho_0-\rho_1).
$
Using~\eqref{eq:initial-data} and~\eqref{eq:ricci-initial}, we have $\rho_1'(0)=0$. Differentiating once more,
\[
\rho_1''
=
a'(\rho_0-\rho_1)
+
a(\rho_0'-\rho_1'),
\]
and therefore
\[
\rho_1''(0)
=
a'(0)\bigl(\rho_0(0)-\rho_1(0)\bigr)
=
1.
\]
Consequently,
\[
\rho_1(t)=\frac12t^2+o(t^2).
\]
In particular, $\rho_1(t)>0$ for $0<|t|<\varepsilon$ after decreasing $\varepsilon$ if necessary. Since $\rho_0(0)=1$ and $\rho_3(0)=3$, continuity also gives $\rho_0(t)>0$ and $\rho_3(t)>0$ for $|t|<\varepsilon$, after one further decrease of $\varepsilon$.

We have therefore proved the following.

\begin{proposition}\label{prop:4d-counterexample}
There exists $\varepsilon>0$ and a smooth metric
\[
g=dt^2+f(t)^2g_{\mathbb S^2}+h(t)^2d\theta^2
\]
on $(-\varepsilon,\varepsilon)\times\mathbb S^2\times\mathbb S^1$ such that $\delta\Rm=0$, $\Ric\geq0$, and
\[
\rank\Ric
=
\begin{cases}
2,&t=0,\\
4,&0<|t|<\varepsilon.
\end{cases}
\]
In particular, the Ricci tensor does not have constant rank.
\end{proposition}

\begin{remark}\label{rem:ric2-fails}
The example also shows explicitly how ordinary nonnegative Ricci curvature permits a cancellation which is ruled out by intermediate Ricci curvature. At $t=0$, formulas~\eqref{eq:K0A}--\eqref{eq:KA3} give
\[
K_{01}=K_{02}=-1,
\qquad
K_{12}=1,
\qquad
K_{03}=3,
\qquad
K_{13}=K_{23}=0.
\]
Hence
\[
\Ric(e_1,e_1)
=
K_{10}+K_{12}+K_{13}
=
-1+1+0
=
0,
\]
while $K_{10}+K_{13}=-1<0$. Thus $\Ric\geq0$ but $\Ric_2\not\geq0$. This is precisely the cancellation mechanism excluded by Lemma~\ref{lem:intermediate-algebra}.
\end{remark}

We now prove Theorem~\ref{thm:counterexample}.

\begin{proof}[Proof of Theorem~\ref{thm:counterexample}]
For $n=4$, the conclusion is Proposition~\ref{prop:4d-counterexample}.

Let $n>4$, and let $(M^4,g)$ be the four-dimensional example above. Consider
\[
(\widehat M^n,\widehat g)
=
(M^4,g)\times(\mathbb R^{n-4},g_{\mathrm{Euc}}).
\]
The Ricci tensor of the product is $\Ric_{\widehat g}=\Ric_g\oplus0$. Since $\Ric_g$ is Codazzi and the Euclidean factor is flat, $\Ric_{\widehat g}$ is again Codazzi. Hence $\delta\Rm_{\widehat g}=0$, and furthermore $\Ric_{\widehat g}\geq0$.

At points with $t=0$, $\rank\Ric_{\widehat g}=2$, whereas for $0<|t|<\varepsilon$, $\rank\Ric_{\widehat g}=4$. Thus the Ricci tensor does not have constant rank. This proves the theorem in every dimension $n\geq4$.
\end{proof}

The same examples also demonstrate directly the sharpness of the index range in Theorem~\ref{thm:mainA}. Indeed, at a point with $t=0$, fix the direction $e_1$ tangent to $\mathbb S^2$. In the product example, the sectional curvatures of the planes spanned by $e_1$ and the additional Euclidean directions all vanish. Choosing the $n-2$ directions consisting of $e_0$, $e_3$, and the $n-4$ Euclidean directions, we obtain
\[
K(e_1,e_0)+K(e_1,e_3)
+
\sum_{\alpha=1}^{n-4}
K(e_1,\partial_{x_\alpha})
=
-1.
\]
Hence $\Ric_{n-2}\not\geq0$ for this example, even though $\Ric_{n-1}=\Ric\geq0$.  So going from \(k=n-2\) to \(k=n-1\) is the real borderline case.

\section{Constant rank theorem for gradient Ricci solitons}
\label{sec:ricci-soliton}

We now turn to a second geometric setting in which the constant rank
principle under intermediate Ricci curvature naturally appears. Recall that
a gradient Ricci soliton is a triple $(M^n,g,f)$ satisfying
\begin{equation}\label{eq:soliton-equation}
\Ric+\nabla^2 f=\lambda g
\end{equation}
for some constant $\lambda\in\mathbb R$. According as $\lambda>0$,
$\lambda=0$, or $\lambda<0$, the soliton is called shrinking, steady, or
expanding, respectively.

Under nonnegative sectional curvature, Petersen--Wylie \cite{PetersenWylieClassification} and Guan--Lu--Xu
\cite{GuanLuXu} proved that the Ricci tensor of a gradient Ricci soliton
has constant rank and in the complete case, its universal cover
splits isometrically along the Ricci-null directions. In this section we
show that the sectional curvature assumption can again be weakened to
nonnegative intermediate Ricci curvature $\Ric_k\geq 0, (1\leq k\leq n-2)$.

We first recall the standard elliptic equation satisfied by the Ricci
tensor of a gradient Ricci soliton. Let
$
\Delta_f=\Delta-\nabla_{\nabla f}
$
denote the weighted Laplacian. Then
\begin{equation}\label{eq:weighted-ricci}
(\Delta_f\Ric)_{ij}
=
2\lambda R_{ij}
-
2R_{ikjl}R_{kl}.
\end{equation}
See, for instance, \cite{PetersenWylieClassification}.

By Lemma~\ref{lem:intermediate-algebra}, the assumption
$\Ric_k\geq0$, $1\leq k\leq n-2$, implies
$
\Ric\geq0
$
and, for every orthonormal pair $X,Y$,
\begin{equation}\label{eq:soliton-sectional-estimate}
K(X,Y)
\geq
-c_{n,k}\Ric(X,X),
\qquad
c_{n,k}:=\frac{k-1}{n-1-k}.
\end{equation}

Let
\[
0\leq\mu_1\leq\cdots\leq\mu_n
\]
be the eigenvalues of the Ricci tensor at a fixed point and choose a
corresponding orthonormal eigenbasis $\{e_i\}$. From
\eqref{eq:weighted-ricci},
\[
(\Delta_f\Ric)_{ii}
=
2\lambda\mu_i
-
2\sum_jK_{ij}\mu_j.
\]
Using \eqref{eq:soliton-sectional-estimate} and the nonnegativity of the
Ricci eigenvalues, we obtain
\[
-\sum_jK_{ij}\mu_j
\leq
-\sum_{K_{ij}<0}K_{ij}\mu_j
\leq
c_{n,k}\mu_i
\sum_{K_{ij}<0}\mu_j
\leq
c_{n,k}R\mu_i.
\]
Consequently,
\begin{equation}\label{eq:weighted-component-estimate}
(\Delta_f\Ric)_{ii}
\leq
\bigl(2\lambda+2c_{n,k}R\bigr)\mu_i.
\end{equation}

For $1\leq\ell\leq n$, set
$
\sigma_\ell=\mu_1+\cdots+\mu_\ell.
$
We first derive the weighted analogue of
Proposition~\ref{prop:subtrace-inequality}.

\begin{proposition}\label{prop:soliton-subtrace}
The lower Ricci subtrace $\sigma_\ell$ satisfies
\begin{equation}\label{eq:soliton-subtrace}
\Delta_f\sigma_\ell
\leq
\bigl(2\lambda+2c_{n,k}R\bigr)\sigma_\ell
\end{equation}
in the viscosity sense.
\end{proposition}

\begin{proof}
Fix $p\in M$, and let $\xi$ be a smooth lower support for
$\sigma_\ell$ at $p$. Choose an orthonormal Ricci eigenbasis
$e_1,\ldots,e_n$ at $p$ corresponding to the eigenvalues
$\mu_1\leq\cdots\leq\mu_n$. By Lemma~\ref{lem:subtrace-viscosity},
\[
\Delta\xi(p)
\leq
\sum_{i=1}^{\ell}(\Delta\Ric)(e_i,e_i),
\]
while \eqref{eq:subtrace-gradient} gives
\[
\nabla_X\xi(p)
=
\sum_{i=1}^{\ell}
(\nabla_X\Ric)(e_i,e_i)
\]
for every $X\in T_pM$. Taking $X=\nabla f$, we therefore obtain
\[
\Delta_f\xi(p)
\leq
\sum_{i=1}^{\ell}
(\Delta_f\Ric)(e_i,e_i).
\]
Applying \eqref{eq:weighted-component-estimate} yields
\[
\begin{aligned}
\Delta_f\xi(p)
&\leq
\bigl(2\lambda+2c_{n,k}R(p)\bigr)
\sum_{i=1}^{\ell}\mu_i\\
&=
\bigl(2\lambda+2c_{n,k}R(p)\bigr)
\sigma_\ell(p),
\end{aligned}
\]
which proves \eqref{eq:soliton-subtrace}.
\end{proof}

\begin{proof}[Proof of the constant rank assertion in
Theorem~\ref{thm:mainB}]
Let
$
m=\max_{x\in M}\dim\ker\Ric_x.
$
If $m=0$, then $\Ric>0$ everywhere and there is nothing to prove.
Suppose $m\geq1$ and choose $p\in M$ such that
$\dim\ker\Ric_p=m$. Then
$
\sigma_m(p)=0,
$ and $
\sigma_m\geq0.
$

Let $\Omega$ be a connected relatively compact neighborhood
of $p$. Since the scalar curvature is smooth, there exists a constant
$C_\Omega\geq0$ such that
\[
2\lambda+2c_{n,k}R\leq C_\Omega
\]
on $\Omega$. Proposition~\ref{prop:soliton-subtrace} gives
\[
\Delta_f\sigma_m-C_\Omega\sigma_m\leq0
\]
in the viscosity sense on $\Omega$. Since $\sigma_m\geq0$ and
$\sigma_m(p)=0$, the strong maximum principle implies
$\sigma_m\equiv0$ on $\Omega$.

It follows that the zero set of $\sigma_m$ is both open and closed.
Since $M$ is connected and nonempty, we conclude that
\[
\sigma_m\equiv0
\]
on $M$. Hence $\dim\ker\Ric_x\geq m$ for every $x\in M$. By the
definition of $m$, equality holds everywhere, and therefore the Ricci
tensor has constant rank.
\end{proof}

We next prove that the Ricci-null distribution is parallel. This part
uses only the algebraic consequence of intermediate Ricci curvature and
the weighted Ricci equation.

Since the Ricci tensor has constant rank,
$
\mathcal N=\ker\Ric
$
is a smooth subbundle. If $V\in\mathcal N$, then
Corollary~\ref{cor:null-jacobi} gives
\begin{equation}\label{eq:soliton-null-jacobi}
R(V,\cdot,V,\cdot)\equiv0.
\end{equation}

\begin{proposition}\label{prop:soliton-nullity-parallel}
Under the assumptions of Theorem~\ref{thm:mainB}, the
distribution $\mathcal N=\ker\Ric$ is parallel.
\end{proposition}

\begin{proof}
Let $V$ be a local smooth section of $\mathcal N$. Then
$\Ric(V,\cdot)=0$, and in particular
$
\Ric(V,V)\equiv0.
$ Applying $\Delta_f$ to this identity gives, at a point where
$\{e_a\}$ is a normal orthonormal frame,
\begin{align}
0
&=
\Delta_f\bigl(\Ric(V,V)\bigr)\nonumber\\
&=
(\Delta_f\Ric)(V,V)
+
4\sum_a
(\nabla_{e_a}\Ric)(\nabla_{e_a}V,V)\nonumber\\
&\qquad
+
2\Ric(\Delta_fV,V)
+
2\sum_a
\Ric(\nabla_{e_a}V,\nabla_{e_a}V).
\label{eq:soliton-nullity-bochner}
\end{align}
Since $V\in\ker\Ric$, we have
$\Ric(\Delta_fV,V)=0$. Moreover,
$\Ric(\nabla_{e_a}V,V)=0$. Differentiating the latter identity and
summing over $a$ yields
\[
\sum_a
(\nabla_{e_a}\Ric)(\nabla_{e_a}V,V)
=
-
\sum_a
\Ric(\nabla_{e_a}V,\nabla_{e_a}V).
\]
Thus \eqref{eq:soliton-nullity-bochner} becomes
\begin{equation}\label{eq:soliton-nullity-bochner2}
0
=
(\Delta_f\Ric)(V,V)
-
2\sum_a
\Ric(\nabla_{e_a}V,\nabla_{e_a}V).
\end{equation}

On the other hand, by \eqref{eq:weighted-ricci},
\[
(\Delta_f\Ric)(V,V)
=
2\lambda\Ric(V,V)
-
2R(V,e_i,V,e_j)R_{ij}.
\]
Both terms on the right-hand side vanish: the first because
$V\in\ker\Ric$, and the second by
\eqref{eq:soliton-null-jacobi}. Therefore
\[
(\Delta_f\Ric)(V,V)=0.
\]
It follows from \eqref{eq:soliton-nullity-bochner2} that
\[
\sum_a
\Ric(\nabla_{e_a}V,\nabla_{e_a}V)=0.
\]
Since $\Ric\geq0$, every summand is nonnegative. Hence
\[
\Ric(\nabla_XV,\nabla_XV)=0
\]
for every vector field $X$, and therefore
$\nabla_XV\in\ker\Ric$. Thus $\mathcal N$ is parallel.
\end{proof}

We can now complete the proof of Theorem~\ref{thm:mainB}.

\begin{proof}[Proof of the splitting statement in
Theorem~\ref{thm:mainB}]
Assume that $(M,g)$ is complete. Since $\mathcal N$ is parallel, its
lift to the universal cover is parallel, as is its orthogonal
complement. By the de Rham decomposition theorem,
\[
(\widetilde M,\widetilde g)
=
(M_0^m,g_0)\times(N^{n-m},g_N),
\]
where $m=\dim\ker\Ric$.

For $V\in TM_0$, \eqref{eq:soliton-null-jacobi} implies that every
sectional curvature containing $V$ vanishes. In particular, $M_0$ is
flat. Since it is complete and simply connected,
$
(M_0,g_0)\cong\mathbb R^m.
$
Moreover, the kernel of the Ricci tensor is precisely the tangent bundle
of the Euclidean factor, and hence
$
\Ric_{g_N}>0.
$

It remains to describe the potential. Let $\widetilde f$ be the lift of
$f$ to $\widetilde M=\mathbb R^m\times N$. The mixed components of the
soliton equation give
\[
\nabla^2\widetilde f(X,Y)=0, \quad  X\in T\mathbb R^m,\quad Y\in TN.
\]
Hence
\[
\widetilde f(x,y)=f_0(x)+f_N(y)+C.
\]
Restricting \eqref{eq:soliton-equation} to the Euclidean factor gives
$
\nabla_{\mathbb R^m}^2f_0=\lambda g_{\mathrm{Euc}},
$
and therefore
$
f_0(x)=\frac{\lambda}{2}|x|^2+\langle a,x\rangle.
$
Restricting the soliton equation to $N$ gives
\[
\Ric_{g_N}+\nabla_N^2f_N=\lambda g_N.
\]
This proves \eqref{eq:potential-splitting} and completes the proof.
\end{proof}

\begin{remark}
When $\lambda\neq0$, a translation in the Euclidean factor removes the
linear term $\langle a,x\rangle$ in
\eqref{eq:potential-splitting}. Thus, up to translation, the Euclidean
factor carries the standard Gaussian shrinking or expanding soliton
potential $\lambda|x|^2/2$. In the steady case $\lambda=0$, the
Euclidean component of the potential is affine.
\end{remark}

Lastly, we show that the range $1\leq k\leq n-2$ in
Theorem~\ref{thm:mainB} is sharp, i.e., the ordinary nonnegative
Ricci curvature does not imply the constant rank property.

We first construct a four-dimensional example. Let
$
M_\varepsilon
=
(-\varepsilon,\varepsilon)\times\mathbb S^2\times\mathbb S^1
$
and consider
\[
g
=
dt^2+F(t)^2g_{\mathbb S^2}
+H(t)^2d\theta^2.
\]
where $g_{\mathbb S^2}$ denotes the round metric of Gaussian
curvature $1$. We look for a soliton potential $\phi=\phi(t)$ and set
\[
a=\frac{F'}F,
\qquad
b=\frac{H'}H,
\qquad
p=\phi'.
\]

Let $e_0= \frac{\partial}{\partial t}$, let $e_1,e_2$ be tangent to the sphere factor,
and let $e_3$ be tangent to the circle factor. Since the curvature computation is identical to the one in Section~2, we only record the resulting formulas.
The nonzero sectional
curvatures are
\[
K_{0A}=-(a'+a^2),\qquad
K_{12}=\frac1{F^2}-a^2,\qquad
K_{03}=-(b'+b^2),\qquad
K_{A3}=-ab,
\]
where $A=1,2$. Hence the three distinct Ricci eigenvalues, with
$\rho_1$ of multiplicity two, are
\begin{align}
\rho_0&=-2(a'+a^2)-(b'+b^2),\label{eq:shrinker-rho0}\\
\rho_1&=\frac1{F^2}-a'-2a^2-ab,\notag\\
\rho_3&=-b'-b^2-2ab.\label{eq:shrinker-rho3}
\end{align}

Since $\phi$ depends only on $t$,
\[
\nabla^2\phi(e_0,e_0)=p',
\qquad
\nabla^2\phi(e_A,e_A)=ap,
\qquad
\nabla^2\phi(e_3,e_3)=bp.
\]
Therefore the normalized shrinking soliton equation
$
\Ric+\nabla^2 \phi=\frac{1}{2} g
$
is equivalent to
\begin{equation}\label{eq:shrinker-eigen-equations}
\rho_0+p'= \frac{1}{2},\qquad
\rho_1+ap=\frac{1}{2},\qquad
\rho_3+bp=\frac{1}{2}.
\end{equation}

Substituting \eqref{eq:shrinker-rho0}--\eqref{eq:shrinker-rho3} into
\eqref{eq:shrinker-eigen-equations}, we obtain the regular first-order
system
\begin{equation}\label{eq:shrinker-ODE}
\left\{
\begin{aligned}
F'&=aF,\\
H'&=bH,\\
a'&=\frac1{F^2}-2a^2-ab+ap-\frac{1}{2},\\
b'&=-b^2-2ab+bp-\frac{1}{2},\\
p'&=\frac2{F^2}-2a^2-4ab+2ap+bp-1.
\end{aligned}
\right.
\end{equation}
We prescribe the initial data
\[
F(0)=\frac{1}{\sqrt{2}},\qquad
H(0)=1,\qquad
a(0)=-\frac{3}{\sqrt{2}},\qquad
b(0)=\frac{1}{\sqrt{2}},\qquad
p(0)=\frac{1}{\sqrt{2}}.
\]
Since the right-hand side of \eqref{eq:shrinker-ODE} is smooth for
$F,H>0$, standard ODE theory provides a unique smooth local solution
on $(-\varepsilon,\varepsilon)$, after decreasing $\varepsilon>0$ if
necessary.

At $t=0$, equations \eqref{eq:shrinker-ODE} give
\[
a'(0)=-\frac{15}{2},\qquad
b'(0)=\frac{5}{2},\qquad
p'(0)=-\frac{5}{2}.
\]
Using \eqref{eq:shrinker-eigen-equations}, we therefore obtain
\[
\rho_0(0)=3,\qquad
\rho_1(0)=2,\qquad
\rho_3(0)=0.
\]
Thus the Ricci eigenvalues at $t=0$ are
\[
3,\quad2,\quad2,\quad0,
\]
and in particular
$
\rank\Ric|_{t=0}=3.
$

It remains to verify that the zero eigenvalue becomes positive on both
sides of $t=0$. From \eqref{eq:shrinker-eigen-equations},
$
\rho_3=\frac{1}{2}-bp,
$
and hence
$
\rho_3'=-(b'p+bp').
$
The initial data give $\rho_3'(0)=0$. Differentiating the last two
equations in \eqref{eq:shrinker-ODE} and evaluating at $t=0$ yields
\[
b''(0)=\frac{25\sqrt{2}}{2},
\qquad
p''(0)=-3\sqrt{2}.
\]
Consequently,
\[
\begin{aligned}
\rho_3''(0)
&=-(bp)''(0)\\
&=-\bigl(b''p+2b'p'+bp''\bigr)\big|_{t=0}\\
&=-\bigl( \frac{25}{2}-\frac{25}{2}-3\bigr)=
3.
\end{aligned}
\]
Therefore
\[
\rho_3(t)=\frac{3}{2}t^2+o(t^2).
\]
and hence $\rho_3(t)>0$ for $0<|t|\ll1$. Since
$\rho_0(0)=3$ and $\rho_1(0)=2$, continuity shows, after possibly
decreasing $\varepsilon$, that
\[
\rho_0>0,\qquad
\rho_1>0,\qquad
\rho_3\geq0
\]
throughout $M_\varepsilon$. We have proved:

\begin{proposition}\label{prop:shrinker-4d-counterexample}
There exists a four-dimensional gradient shrinking Ricci soliton
$(M_\varepsilon,g,\phi)$ satisfying
\[
\Ric+\nabla^2\phi=\frac{1}{2}g,
\qquad
\Ric\geq0,
\]
such that
\[
\rank\Ric=
\begin{cases}
3,&t=0,\\
4,&0<|t|<\varepsilon.
\end{cases}
\]
\end{proposition}

The example also makes clear why the hypothesis $\Ric_2\geq0$ rules
out this phenomenon. Indeed, at $t=0$ we have
\[
K_{01}=K_{02}=3,\qquad
K_{12}=-\frac{5}{2},\qquad
K_{03}=-3,\qquad
K_{13}=K_{23}=\frac{3}{2}.
\]
Thus
\[
\Ric(e_3,e_3)=K_{30}+K_{31}+K_{32}=-3+\frac{3}{2}+\frac{3}{2}=0,
\]
whereas
\[
K(e_3,e_0)+K(e_3,e_1)=-3+\frac{3}{2}=-\frac{3}{2}<0.
\]
Hence the example satisfies $\Ric\geq0$ but not $\Ric_2\geq0$.

\begin{proof}[Proof of Theorem~\ref{thm:soliton-sharpness}]
The case $n=4$ follows from
Proposition~\ref{prop:shrinker-4d-counterexample}. For $n>4$, take the
product of this soliton with the Gaussian shrinking soliton
\[
\left(
\mathbb R^{n-4},
g_{\mathrm{Euc}},
\frac{|z|^2}{4}
\right).
\]
Thus
\[
\widehat M
=
M_\varepsilon\times\mathbb R^{n-4},
\qquad
\widehat g
=
g+g_{\mathrm{Euc}},
\qquad
\widehat\phi
=
\phi+\frac{|z|^2}{4}.
\]
Then
\[
\Ric_{\widehat g}+\nabla^2\widehat\phi
=
 \frac{1}{2} \widehat g,
\]
and
\[
\Ric_{\widehat g}
=
\Ric_g\oplus0
\geq0.
\]
Its rank equals $3$ along $t=0$ and $4$ for
$0<|t|<\varepsilon$, so it is not constant.

Finally, the same examples show directly that the endpoint
$\Ric_{n-1}=\Ric\geq0$ cannot replace the assumption
$\Ric_k\geq0$, $k\leq n-2$. At $t=0$, fix the direction $e_3$.
The sectional curvatures involving $e_3$ are
\[
-3,\quad\frac{3}{2},\quad\frac{3}{2},\quad
\underbrace{0,\ldots,0}_{n-4}.
\]
Among the $n-1$ orthogonal directions, choose the $n-2$ directions
obtained by omitting one of the two directions with curvature $\frac{3}{2}$.
Their sum is $-\frac{3}{2}$. Hence
\[
\Ric_{n-2}\not\geq0.
\]
This proves both the theorem and the sharpness of the range
$1\leq k\leq n-2$.
\end{proof}

\begin{remark}
The counterexample above is local, since the constant-rank conclusion
in Theorem~\ref{thm:mainB} is itself local and does not require
completeness. It is a separate question whether there exists a complete
gradient Ricci soliton with $\Ric\geq0$ and nonconstant Ricci rank. For recent related work, see \cite{Wu-Wu2}.
\end{remark}

{\bf Acknowledgment.} Both authors would like to thank Professor Xi-Nan Ma for his helpful discussion.

{\bf Declaration of AI use}
 The authors used ChatGPT (OpenAI) during the preparation of the manuscript for language editing and exploratory calculations. All mathematical arguments and computations were independently checked by the authors.

\end{document}